\documentclass[11pt, reqno]{amsart}
\usepackage{amsaddr}
\usepackage{latexsym,amsfonts}
\usepackage{amssymb,amsthm,upref,amscd}
\usepackage[T1]{fontenc}
\usepackage{times}
\usepackage{amscd}
\usepackage{enumerate}
\usepackage{mathrsfs}

\usepackage{amsmath}
\usepackage{color}
\usepackage{soul}
\usepackage{indentfirst}
\usepackage{comment}
\PassOptionsToPackage{reqno}{amsmath}
\usepackage{amsmath,amsfonts,amsthm,amssymb,bbm}
\usepackage{graphicx,color,dsfont}
\usepackage{enumitem}
\usepackage{fourier}

\usepackage{mathtools}

\newtheorem{thm}{Theorem}[section]

\newtheorem{lem}{Lemma}[section]
\newtheorem{cor}{Corollary}[section]

\newtheorem{rem}{Remark}[section]

\newcommand{\R}{\mathbb{R}}

\numberwithin{equation}{section}

\newcommand{\wto}{\rightharpoonup}

\makeatletter \@addtoreset{equation}{section} \makeatother

\usepackage[
left=30mm,
right=30mm,
top=25.4mm,
bottom=25.4mm,
headheight=2.17cm,
headsep=4mm,
footskip=12mm,
heightrounded]{geometry}

\usepackage[colorlinks=true,urlcolor=blue,
citecolor=black,linkcolor=black,linktocpage,pdfpagelabels,
bookmarksnumbered,bookmarksopen]{hyperref}

\newcounter{const}
\author[T. Gou]{Tianxiang Gou}
\address[T. Gou]{
\centerline{School of Mathematics and Statistics, Xi’an Jiaotong University,}
\centerline{710049, Xi’an, Shaanxi, China}}



\email{tianxiang.gou@xjtu.edu.cn}

\title[DNLS]{Radial symmetry, uniqueness and non-degeneracy of solutions to degenerate nonlinear Schr\"odinger equations}
\thanks{{\it Acknowledgments.} T. Gou was supported by the National Natural Science Foundation of China (No. 12471113). The author warmly thanks Profs. Rupert Frank and Jean Dolbeault for useful discussions and for bringing the literature \cite{A, DEM} to his attention. He also thanks Prof. Rupert Frank for pointing out a mistake presented in the original proof of Theorem 2.1. {The author also would like to thank warmly the knowledgeable referee for useful remarks.}}

\thanks{{\it Conflict of interest statement}. The author declares that there is no conflict of interest.}

\thanks{{\it Data availability statement}. The author affirms that the paper has no associated data.}

\begin{document}

\begin{abstract} 
In this paper, we consider the radial symmetry, uniqueness and non-degeneracy of solutions to the degenerate nonlinear elliptic equation
$$
-\nabla \cdot \left(|x|^{2a} \nabla u\right) + \omega u=|u|^{p-2}u \quad \mbox{in} \,\, \R^d,
$$
where $d \geq 2$, $0<a<1$, $\omega>0$ and $2<p<\frac{2d}{d-2(1-a)}$. We proved that any ground state is radially symmetric and strictly decreasing in the radial direction. Moreover, we establish the uniqueness of ground states and derive the non-degeneracy of ground states in the corresponding radially symmetric Sobolev space. This affirms the {natural} conjectures posed recently in \cite{IS}.

\end{abstract}

\maketitle

\thispagestyle{empty}

\section{Introduction}

In this paper, we investigate quantitative properties of solitary wave solutions to the degenerate nonlinear Schr\"odinger equation (NLS)
\begin{align} \label{equt}
\textnormal{i} \partial_t \psi + \nabla \cdot \left(|x|^{2a} \nabla \psi\right) + |\psi|^{p-2} \psi=0 \quad \mbox{in} \,\, \R^d,
\end{align}
where $d \geq 2$, $0<a<1$ and $2<p<2^*_a:=\frac{2d}{d-2(1-a)}$. Here a solitary wave solution $\psi$ to \eqref{equt} is of the form 
$$
\psi(t, x)=e^{\textnormal{i} \omega t} u(x), \quad \omega \in \R.
$$
This implies that $u$ satisfies the profile equation
\begin{align} \label{equ}
-\nabla \cdot \left(|x|^{2a} \nabla u\right) + \omega u=|u|^{p-2}u \quad \mbox{in} \,\, \R^d.
\end{align}
Nonlinear Schr\"odinger equations driven by the operator $\textnormal{i} \partial_t \psi + \nabla \cdot \left(\sigma(x)\nabla \psi\right)$ with degenerate $\sigma(x)$ appear in various physical processes, such as plasma physics and the research of non-equilibrium magnetism, see \cite{DPL, DL, I} and references therein.

The aim of the paper is to further study solutions to \eqref{equ}. For this, we shall present some useful facts. Hereafter, we denote by $H^{1,a}(\R^d)$ the Sobolev space defined by the completion of $C^{\infty}_0(\R^d)$ under the norm
$$
\|u\|_{H^{1,a}}:=\left(\int_{\R^d} |x|^{2a}|\nabla u|^2 + |u|^2 \,dx \right)^{\frac 12}.
$$
And we denote by $H^{1,a}_{rad}(\R^d)$ the subspace of $H^{1,a}(\R^d)$ consisting of radially symmetric functions in $H^{1,a}(\R^d)$. From the classical Caffarelli-Kohn-Nirenberg's inequalities in the celebrated work \cite{CKN}, we know that, for any $u \in H^{1,a}(\R^d)$,
\begin{align} \label{embedding}
\left(\int_{\R^d}|u|^q \,dx \right)^{\frac 1q} \lesssim \left( \int_{\R^d} |\nabla u|^2 |x|^{2a} \,dx\right)^{\frac{\theta}{2}}  \left( \int_{\R^d} |u|^2\,dx\right)^{\frac{1-\theta}{2}},
\end{align}
where 
$$
0 \leq \theta \leq 1, \quad \frac 1q=\frac 12-\frac{\theta(1-a)}{d}.
$$
Throughout the paper, we shall use the notation $X \lesssim Y$ to denote $X \leq CY$ for some proper constant $C>0$. As a consequence, by utilizing \eqref{embedding}, one gets that $H^{1,a}(\R^d)$ is embedded continuously into $L^q(\R^d)$ for $2 \leq q \leq 2_a^*$, where $L^{q}(\R^d)$ denotes the usual Lebesgue space equipped with the norm
$$
\|u\|_q:=\left(\int_{\R^d} |u|^q \,dx \right)^{\frac 1q}, \quad 1 \leq q<+\infty.
$$ 
Moreover, by using \cite[Proposition 2]{IS}, one knows that $H^{1,a}(\R^d)$ is embedded compactly into $L^q(\R^d)$ for $2 < q < 2_a^*$. 

When $\omega>0$, to derive the existence and spectral stability of {standing wave solutions} to \eqref{equ}, Iyer and Stefanov in \cite{IS} introduced the following minimization problem,
\begin{align} \label{min10}
\widetilde{m}:=\inf_{u \in H^{1,a}(\R^d) \backslash \{0\}} J[u],
\end{align}
where
\begin{align*} 
J[u]:=\frac{\int_{\R^d} |\nabla u|^2 |x|^{2a} \,dx + \omega \int_{\R^d} |u|^2 \,dx}{\left(\int_{\R^d} |u|^p \,dx\right)^{\frac 2 p}}.
\end{align*}
Applying the compact embedding in $H^{1,a}(\R^d)$, they obtained the compactness of any minimizing sequence to \eqref{min10}, which then leads to the existence of {the solutions} to \eqref{equ}. In addition, they established asymptotic behaviors of any positive solution in $H^1_a(\R^d)$ under the radial symmetry assumption. More precisely, they proved the following interesting result.

\begin{thm} \label{thmdecay} (\cite[Theorem 1]{IS})
Let $d \geq 2$, $0<a<1$, $\omega>0$ and $2<p<2^*_a$. Then the following assertions hold true.
\begin{itemize}
\item[$(\textnormal{i})$] There exists a {distributional positive solution} $u \in H^{1,a}(\R^d) \cap C^{\infty}(\R^d \backslash \{0\})$ to \eqref{equ}. 
\item[$(\textnormal{ii})$] If $u \in H^{1,a}(\R^d)$ is a positive solution to \eqref{equ}, then it satisfies the pointwise exponential bound
$$
0<u(x) \lesssim e^{-\delta |x|^{1-a}}, \quad x \in \R^d
$$
for some $\delta>0$. In case that the solution $u$ is radially symmetric, then it is continuous at zero and satisfies that
\begin{align*} 
u'(r)=-\frac{u^{p-1}(0)-\omega u(0)}{d} r^{1-2a} +o(r^{1-2a}), \quad r \to 0,
\end{align*}
$$
u''(r)=\frac{(2a-1) \left(u^{p-1}(0)-\omega u(0)\right)}{d} r^{-2a} +o(r^{-2a}), \quad r \to 0,
$$
where $u^{p-1}(0)-\omega u(0)>0$. In particular, if $0<a \leq 1/2$, then $u \in C^1(0, +\infty)$, if $1/2<a<1$, then $u'(r)$ blows up as $r \to 0^+$. Meanwhile, $u''$ always blows up at zero for any $0<a<1$.
\end{itemize}
\end{thm}

\begin{rem}
In fact, while $\omega \leq 0$, by using Pohozaev's identities satisfied by solutions to \eqref{equ} (see \cite[Proposition 1]{IS}), one finds easily that there exists no {solutions} to \eqref{equ} in $H^{1,a}(\R^d)$.
\end{rem}

Note that, in Theorem \ref{thmdecay}, the asymptotic behaviors of the solution hold true provided that it is radially symmetric. It is conjectured in \cite{IS} that the solution to \eqref{equ} is indeed radially symmetric.
Furthermore, in \cite{IS}, Iyer and Stefanov conjectured that ground states to \eqref{equ} are unique and non-degenerate. {Here we say that $u \in H^{1,a}(\R^d) \backslash \{0\}$ is a ground state to \eqref{equ} if it holds that
$$
E[u]=\inf_{\phi \in \mathcal{N}} E[\phi],
$$
where $E$ is the underlying energy functional and $\mathcal{N}$ is the so-called Nehari manifold associated to \eqref{equ} defined respectively by
$$
E[\phi]:=\frac 12 \int_{\R^d} |\nabla \phi|^2 |x|^{2a} \,dx + \frac 12 \int_{\R^d} |\phi|^2 \,dx -\frac 1p \int_{\R^d} |\phi|^p \,dx,
$$
$$
\mathcal{N}:=\left\{\phi \in H^{1,a}(\R^d) \backslash\{0\} : \langle E'[\phi], \phi \rangle=0\right\}.
$$
} 
In this present paper, we shall give affirmative answers to the conjectures. And our results read as follows.

\begin{thm} \label{mainthm}
Let $d \geq 2$, $0<a<1$, $\omega>0$ and $2<p<2^*_a$. Then the following assertions hold true.
\begin{itemize}
\item [$(\textnormal{i})$] Any ground state to \eqref{equ} is radially symmetric and strictly decreasing in the radial direction.
\item [$(\textnormal{ii})$] There exists only one ground state to \eqref{equ} in $H^{1,a}(\R^d)$.
\item [$(\textnormal{iii})$]  Any ground state to \eqref{equ} is non-degenerate in $H^{1,a}_{rad}(\R^d)$.
\end{itemize}
\end{thm}

\begin{rem}
From Theorem \ref{mainthm}, we then conclude that Theorem \ref{thmdecay} remains true when the {radial symmetry} assumption is removed.
\end{rem}

To prove the radial symmetry of ground states to \eqref{equ}, we are inspired by the classical polarization arguments (see for example \cite{BS}). Let $u \in H^{1,a}(\R^d)$ be a ground state to \eqref{equ}. First, by the maximum principle, we know that it is indeed positive. Then, relying on the polarization arguments, we are able to show that 
\begin{align} \label{r11} 
\int_{\R^d} |\nabla u^*|^2 |x|^{2a} \,dx \leq \int_{\R^d} |\nabla u|^2 |x|^{2a} \,dx,
\end{align}
see Lemma \ref{lemr1}, where $u^*$ denotes the symmetric-decreasing rearrangement of the function $u$. 
Actually, when $a=0$, then \eqref{r11} reduces to the well-known P\'olya–Szeg\"o type inequality (see for example \cite[Theorem 3.20]{BDL} or \cite[Lemma 7.17]{LL})
\begin{align} \label{r110} 
\int_{\R^d} |\nabla u^*|^2 \,dx \leq \int_{\R^d} |\nabla u|^2 \,dx.
\end{align}
Since $u \in H^{1,a}(\R^d)$ is a ground state to \eqref{equ}, depending on the variational characterization of the ground state energy,  we then justify that it is a minimizer to \eqref{min10}. As a consequence, applying \eqref{r11} and the fact (see for example \cite{LL})
$$
\int_{\R^d} |u^*|^q \,dx=\int_{\R^d} |u|^q \,dx,  \quad 1 \leq q<+\infty,
$$
we conclude that $u^* \in H^{1,a}(\R^d)$ is also a minimizer to \eqref{equ} and it holds that
\begin{align} \label{r111} 
\int_{\R^d} |\nabla u^*|^2 |x|^{2a} \,dx = \int_{\R^d} |\nabla u|^2 |x|^{2a} \,dx.
\end{align}
Subsequently, making use of the well-known coarea formula,
we get that
\begin{align} \label{r22}
\int_{(u^*)^{-1}(t)} |x|^{\mu} \, \mathcal{H}_{d-1}(dx) = \int_{u^{-1}(t)} |x|^{\mu} \, \mathcal{H}_{d-1}(dx),
\end{align}
where $\mu\geq 1$ is a constant, $\mathcal{H}_{d-1}$ denotes $(d-1)$ dimensional Hausdorff measure and $u^{-1}(t):=\{x \in \R^d : u(x)=t\}$ for $t>0$. At this point, by invoking \cite[Theorem 4.3]{BBMP}, we then have that
\begin{align} \label{r23}
\left\{ x\in \R^d : u(x) > t\right\} = \left\{ x\in \R^d : u^*(x) > t\right\}.
\end{align}
It indicates immediately that $u=u^*$, i.e $u$ is radially symmetric, because we know that
$$
u(x)=\int_0^{+\infty} \chi_{\left\{ x\in \R^d : u(x) > t\right\}}(x) \,dt, \quad u^*(x)=\int_0^{+\infty} \chi_{\left\{ x\in \R^d : u^*(x) > t\right\}}(x) \,dt,
$$ 
where $\chi_A$ denotes the characteristic function on the set $A$. It is necessary to point out that when $\mu=0$, by employing only \eqref{r22}, one cannot derive that \eqref{r23} is valid, see for instance \cite{BZ}. While $\mu \geq 1$, then the function $\tau \mapsto \tau^{\mu}$ is strictly increasing and convex on $[0, +\infty)$. This jointly with \eqref{r22} then leads to the desired result, see \cite[Theorem 4.3]{BBMP}.

Further, to assert that $u$ is strictly decreasing in the radial direction, we shall assume by contradiction and take into account of the well-posedness of solutions to the initial value problem for the ordinary differential equation
\begin{align}  \label{od1}
u''+\frac{d-1+2a}{r}u'-\frac{u}{r^{2a}}+\frac{u^{p-1}}{r^{2a}}=0.
\end{align}
In fact, since the solution $u$ is radially symmetric, then it solves necessarily \eqref{od1}.

When it comes to demonstrate that \eqref{equ} admits only one ground state in $H^{1,a}(\R^d)$, we shall follow the approach due to 
Yanagida in \cite{Ya}. Let $u \in H^{1,a}(\R^d)$ be a ground state to \eqref{equ}. By the assertion $(\textnormal{i})$ of Theorem \ref{thmdecay}, we know that $u$ is radially symmetric. Now we write $u=u(r)$ for $r=|x|$. It then satisfies the ordinary differential equation
\begin{align} \label{equ100}
\left\{
\begin{aligned}
&u''+\frac{d-1+2a}{r}u'-\frac{u}{r^{2a}}+\frac{u^{p-1}}{r^{2a}}=0, \\
&u(0)>0, \quad \lim_{r \to \infty} u(r)=0.
\end{aligned}
\right.
\end{align}
To adapt the approach to the problem under our consideration, the essence is to introduce the corresponding Pohozaev quantity $J(r, u)$ defined by 
$$
J(r,u):=\frac {A(r)}{2} (u')^2 + B(r) u' u + \frac {C(r)}{2}u^2-\frac{A(r)}{2}\frac{u^2}{r^{2a}} + \frac{A(r)}{p}\frac{u^{p}}{r^{2a}},
$$
where 
$$
A(r)=r^{\frac{4a+2(d-1+2a)p}{p+2}}, 
$$ 
$$
B(r)=\frac{2d-2+2a}{p+2}r^{\frac{4a+2(d-1+2a)p}{p+2}-1},
$$
$$
C(r)=\frac{2d-2+2a}{p+2} \left(d+2a -\frac{4a+2(d-1+2a)p}{p+2}\right)r^{\frac{4a+2(d-1+2a)p}{p+2}-2}.
$$
Later on, we need to check the properties that 
$$
\lim_{r \to 0}J(r, u)=0, \quad \lim_{r \to +\infty}J(r, u)=0,
$$ 
$J(\cdot, u) \not\equiv 0$ and $J(r,u) \geq 0$ for any $r>0$, see Lemmas  \ref{lem2} and \ref{lem3}. Since $u'(r)$ blows up as $r \to 0$ for $1/2<a<1$ by Theorem \ref{thmdecay}, then we are not able to employ directly the arguments in \cite{Ya} to obtain the desired result. 
Actually, in our scenario, Lemma \ref{expdecay}, which reveals the exponential decay of the derivatives of $u$, i.e. there exists $\widetilde{\delta}>0$ such that, for $1 \leq |\alpha| \leq 2$,
$$
|D^{\alpha} u(x)| \lesssim e^{-\widetilde{\delta} |x|^{1-a}}, \quad |x|>R, 
$$
where $R>0$ is a constant, and the asymptotic behavior
\begin{align*} 
u'(r)=-\frac{u^{p-1}(0)-\omega u(0)}{d} r^{1-2a} +o(r^{1-2a}), \quad r \to 0
\end{align*}
play an important role in the verification of the desired properties. At this stage, taking into account the properties and arguing by contradiction, we can reach a contradiction. This then gives the uniqueness of ground states to \eqref{equ}.

To attain that the ground state is non-degenerate in $H^{1,a}(\R^d)$, one needs to check that the kernel of the linearized operator $\mathcal{L}_+$ is trivial in $H^{1,a}(\R^d)$, where
$$
\mathcal{L}_+:=-\nabla \cdot \left(|x|^{2a} \nabla  \right) + \omega -(p-1) u^{p-2}.
$$
For this purpose, one can argue by contradiction that $Ker[\mathcal{L}_+] \neq 0$, i.e. there exists a {nontrivial} $v \in H^{1,a}(\R^d)$ solving the linearized equation
\begin{align} \label{equ100}
-\nabla \cdot \left(|x|^{2a} \nabla v \right) + v=(p-1) u^{p-2}v.
\end{align}
In the sequel, one of the crucial steps is to prove that the solution $v \in H^{1,a}(\R^d)$ to \eqref{equ100} is radially symmetric. To this end, one can take advantage of the classical spherical harmonics decomposition arguments. Let $Y_k^m$ be the eigenfunctions of $-\Delta_{\mathbb{S}^{d-1}}$ with respect to the corresponding eigenvalue $\mu_k=k(k+d-2)$ and the multiplicity of which is $l_k$, that is
$$
-\Delta_{\mathbb{S}^{d-1}} Y_k^m=\mu_k Y_k^m, \quad k \geq 0,
$$ 
$$
l_k=
\begin{pmatrix}
   k \\
d+k-1
\end{pmatrix}
-
\begin{pmatrix}
 k-2 \\
d+k-3
\end{pmatrix}, 
\quad k \geq 2, \quad l_1=d, \quad l_0=1.
$$ 
Then it holds that
\begin{align} \label{shd}
v(x)=\sum_{k=0}^{\infty} \sum_{m=1}^{l_k} v_k^m(r) Y_k^m(\theta),
\end{align}
where $r=|x|$, $\theta=\frac{x}{|x|} \in \mathbb{S}^{d-1}$ for $x \in \R^d \backslash \{0\}$
and
$$
v_k^m(r)=\int_{\mathbb{S}^{d-1}} v(r \theta) Y_k^m(\theta) \, d\theta. 
$$
Since $v \in H^{1,a}(\R^d)$ is a solution to \eqref{equ}, by \eqref{shd}, then $v_k^m$ satisfies the equation
$$
(v_k^m)''+\frac{d-1+2a}{r}(v_k^m)'-\frac{\mu_k}{r^2} v_k^m-\frac{v_k^m}{r^{2a}}+\frac{p-1}{r^{2a}}u^{p-2}v_k^m=0, \quad k \geq 0.
$$
At this point, to derive the radial symmetry of the solution, it suffices to demonstrate that $v_k^m=0$ for any $0 \leq m \leq l_k$ and $k \geq 1$. Define 
$$
\mathcal{L}_{+, k}:=-\partial_{rr}-\frac{d-1+2a}{r}\partial_r+\frac{\mu_k}{r^2}+\frac{1}{r^{2a}}-\frac{p-1}{r^{2a}}u^{p-2}.
$$
It is clear that $\mathcal{L}_{+, 0}<\mathcal{L}_{+, 1}<\mathcal{L}_{+, 2}< \cdots<\mathcal{L}_{+, k}< \cdots$. 
As a consequence, one only needs to show that $0$ is not an eigenvalue of $\mathcal{L}_{+, 1}$, because Morse's index of the ground state is $1$. For this, when $a=0$, arguing by contradiction and adapting the existing methods, one can reach a contradiction and the result follows. While $a \neq 0$, then the existing ones are not suitable to our problem and the situation becomes involved and delicate. This then forces us to discuss the non-degeneracy of the ground state in $H^{1,a}_{rad}(\R^d)$.

To investigate the non-degeneracy of the ground state in the radially symmetric framework, we first need to establish asymptotic behaviors of the solution $v \in H^{1,a}_{rad}(\R^d)$ to \eqref{equ100}, see Lemma \ref{lemdecay}, i.e. 
$$
v(r) \sim 1, \quad v'(r) =o\left(\frac 1 r\right), \quad r \to 0.
$$
Moreover, if $0<a<1/2$, then
$$
v(r) \sim r^{-\frac{d-1}{2}} e^{-\frac{1}{1-a}r^{1-a}}, \quad v'(r) \sim  r^{-\frac{d-1}{2}-a} e^{-\frac{1}{1-a}r^{1-a}} +o\left(\frac{1}{r^{\frac{d-1}{2}+2a}}e^{-\frac{1}{1-a}r^{1-a}}\right), \quad r \to +\infty.
$$ 
Next, we need to assert that if $v \in H^{1,a}_{rad}(\R^d)$ is a nontrivial solution to \eqref{equ100}, then it changes sign only once, see Lemma \ref{sign}. Indeed, the verification of those two results are based principally on ODE techniques. Finally, utilizing the previous results and arguing by contradiction, we are able to reach a contradiction. This completes the proof.

\begin{cor} \label{c}
Let $d \geq 2$, $0<a<1$, $\omega>0$ and $2<p<2^*_a$. Then there exists only one positive radially symmetric solution to \eqref{equ}.
\end{cor}

To establish Corollary \ref{c}, one can follow closely the method we adapted to demonstrate the uniqueness of ground states to \eqref{equ} in Theorem \ref{mainthm} and the proof shall be omitted.

The proof of Theorem \ref{mainthm} is divided into three sections. In Section \ref{radial}, we consider the radial symmetry of ground states to \eqref{equ} and present the proof of the assertion $(\textnormal{i})$ of Theorem \ref{thmdecay}, see Theorem \ref{symmetry}. In Section \ref{uniqueness}, we investigate the uniqueness of ground states to \eqref{equ}, where contains the proof of the assertion $(\textnormal{ii})$ of Theorem \ref{thmdecay}, see Theorem \ref{thm1}. Section \ref{nond} is devote to the study of the non-degeneracy of ground state to \eqref{equ} and the proof of the assertion $(\textnormal{iii})$ of Theorem \ref{thmdecay}, see Theorem \ref{nd}.

For simplicity, we shall always assume that $d \geq 2$, $0<a<1$, $\omega=1$ and $2<p<2^*_a$ in the remaining parts.

\section{Radial Symmetry} \label{radial}

In this section, we are going to discuss the radial symmetry of ground states to \eqref{equ}. To this end, we shall take advantage of the classical polarization arguments developed in \cite{BS}. We denote by $\mathcal{H}$ the set of all half spaces in $\R^d$. And we denote by $\mathcal{H}_0$ the subset of $\mathcal{H}$ corresponding to $d-1$ dimensional Euclidean hyperplanes.
Let $H \in \mathcal{H}$ be a half space and $\sigma_H$ be the reflexion with respect to $\partial H$. The polarization of a measurable function $u : \R^d \to \R$ with respect to $H$ is defined by
\begin{align} \label{defp}
u_H(x):=\left\{
\begin{aligned}
&\max\left\{u(x), u(\sigma_H(x))\right\}, &\quad x \in H,\\
&\min \left\{u(x), u(\sigma_H(x))\right\}, &\quad x \in \R^d \backslash H.
\end{aligned}
\right.
\end{align}


\begin{lem} \label{lemr10}
Let $u \in H^{1,a}(\R^d)$ and $u \geq 0$. Then, for any $H \in \mathcal{H}_0$, it holds that $u_H \in H^{1,a}(\R^d)$ and 
\begin{align} \label{l2}
\int_{\R^d} u_H^2 \,dx=\int_{\R^d} u^2 \,dx,
\end{align}
\begin{align} \label{nablel2}
\int_{\R^d}|\nabla u_{H}|^2  |x|^{2a}\,dx=\int_{\R^d} |\nabla u|^2 |x|^{2a}\,dx.
\end{align}
\end{lem}
\begin{proof}
Define $v:=u \circ \sigma_H$ for $H \in \mathcal{H}_0$. We then write $u_H$ defined by \eqref{defp} as
\begin{align} \label{defuh}
u_H(x)=\left\{
\begin{aligned}
\frac{u(x)+v(x)}{2} + \frac{|u(x)-v(x)|}{2}, &\quad x \in H,\\
\frac{u(x)+v(x)}{2} - \frac{|u(x)-v(x)|}{2}, &\quad x \in \R^d \backslash H.
\end{aligned}
\right.
\end{align}
Since $u \geq 0$, then $v \geq 0$. Therefore, by making a change of variable with $x \mapsto \sigma_H(x)$, we are able to compute that
\begin{align*}
\int_{\R^d} u_H^2\,dx&=\int_{H} u_H^2\,dx + \int_{\R^d \backslash H} u_H^2 \,dx \\
&=\frac 12 \int_{H} u^2\,dx +\frac 12 \int_{H} v^2\,dx + \frac 12 \int_{H} \left|u^2-v^2 \right|\,dx \\
& \quad +\frac 12 \int_{\R^d \backslash H} u^2\,dx +\frac 12 \int_{\R^d \backslash H} v^2\,dx - \frac 12 \int_{\R^d \backslash H} \left|u^2-v^2 \right|\,dx \\
&=\int_{\R^d} u^2\,dx.
\end{align*}
It then yields that \eqref{l2} holds true. Next we shall verify that \eqref{nablel2} holds true. In view of \eqref{defuh}, we are able to calculate that
$$
\nabla u_H(x)=\left\{
\begin{aligned}
&\nabla u(x) , &\quad x \in \left(H \cap \left\{x \in \R^d : u(x) \geq v(x)\right\} \right) \cup \left(\sigma_H(H) \cap \left\{x \in \R^d : u(x) < v(x)\right\}\right),\\
&\nabla v(x), &\quad x \in \left(H \cap \left\{x \in \R^d : u(x) < v(x)\right\} \right) \cup \left(\sigma_H(H) \cap \left\{x \in \R^d : u(x) \geq v(x)\right\}\right).
\end{aligned}
\right.
$$
Moreover, it holds that
$$
\sigma_{H}\left(H \cap \left\{x \in \R^d : u(x) < v(x)\right\}\right)=\sigma_{H}(H) \cap \left\{x \in \R^d : u(x) > v(x)\right\},
$$
$$
\sigma_{H}\left(H \cap \left\{x \in \R^d : u(x) \geq v(x)\right\}\right)=\sigma_{H}(H) \cap \left\{x \in \R^d : u(x)  \leq v(x)\right\}.
$$
Since $H \in \mathcal{H}_0$, then $|\sigma_H(x)|=|x|$. Consequently, by making a change of variable with $x \to \sigma_H(x)$, we conclude that
\begin{align*}
\int_{\R^d} |\nabla u_{H}|^2 |x|^{2a}\,dx&=\int_{H \cap \left\{x \in \R^d : u(x) \geq v(x)\right\}}|\nabla u|^2 |x|^{2a}\,dx + \int_{\sigma_H(H) \cap \left\{x \in \R^d : u(x) < v(x)\right\}} |\nabla u|^2|x|^{2a} \,dx \\
& \quad +\int_{H \cap \left\{x \in \R^d : u(x) < v(x)\right\}}|\nabla v|^2|x|^{2a} \,dx + \int_{\sigma_H(H) \cap \left\{x \in \R^d : u(x) \geq v(x)\right\}} |\nabla v|^2|x|^{2a} \,dx \\
&=\int_{\R^d}|\nabla u|^2  |x|^{2a}\,dx.
\end{align*} 
It implies that \eqref{nablel2} holds true. In particular, we have that $u_H \in H^{1,a}(\R^d)$. This then completes the proof.
\end{proof}

\begin{lem} \label{lemr1}
Let $u \in H^{1,a}(\R^d)$ and $u \geq 0$. Let $u^*$ be the symmetric-decreasing rearrangement of $u$. Then it holds that $u^* \in H^{1,a}(\R^d)$ and 
\begin{align} \label{r1}
\int_{\R^d} |\nabla u^*|^2 |x|^{2a} \,dx \leq \int_{\R^d} |\nabla u|^2 |x|^{2a} \,dx.
\end{align}
\end{lem}
\begin{proof}
Let $u \in H^{1,a}(\R^d)$ and $u \geq 0$.
It then follows from 
\cite{BS} that there exists a sequence $\{H_n\} \subset \mathcal{H}_0$ such that $u_{H_1H_2\cdots H_n} \to u^*$ in $L^2(\R^d)$ as $n \to \infty$. Define $u_n=u_{H_1H_2\cdots H_n}$.
Using Lemma \ref{lemr10}, we derive that $u_n \in H^{1,a}(\R^d)$, $\|u_n\|_2=\|u\|_2$ and  
\begin{align} \label{de}
\|\nabla u_n\|_{L^2(\R^d; |x|^{2a}\,dx)}=\|\nabla u\|_{L^2(\R^d; |x|^{2a}\,dx)}.
\end{align}
Hence $\{u_n\}$ is bounded in $H^{1,a}(\R^d)$ and $u_n \wto u^*$ in $H^{1,a}(\R^d)$ as $n \to \infty$. 
Therefore, by applying the weak lower semi-continuity of the norm, we obtain that
$$
\|\nabla u^*\|_{L^2(\R^d; |x|^{2a}\,dx)} \leq \|\nabla u\|_{L^2(\R^d; |x|^{2a}\,dx)}.
$$
Thus \eqref{r1} holds true and the proof is complete.
\end{proof}


\begin{thm} \label{symmetry}
Let $u \in H^{1,a}(\R^d)$ be a ground state to \eqref{equ}. Then it is radially symmetric and strictly decreasing in the radial direction.
\end{thm}
\begin{proof}
Let us assume that $u \in H^{1,a}(\R^d)$ is a ground state to \eqref{equ} at the ground state energy level $m>0$, namely
$$
m:=\inf_{\phi \in \mathcal{N}} E[\phi],
$$
where 
$$
E[\phi]=\frac 12 \int_{\R^d} |\nabla \phi|^2 |x|^{2a} \,dx + \frac 12 \int_{\R^d} |\phi|^2 \,dx -\frac 1p \int_{\R^d} |\phi|^p \,dx,
$$
$$
\mathcal{N}=\left\{\phi \in H^{1,a}(\R^d) \backslash\{0\} : \langle E'[\phi], \phi \rangle=0\right\}.
$$
It is well-known that the ground state energy $m$ admits the variational characterization
$$
m=E[u]=\inf_{\phi \in H^{1,a}(\R^d)\backslash\{0\}} \max_{t \geq 0} E[t \phi].
$$
Then we are able to calculate that
$$
\max_{t \geq 0}E[t \phi]=\frac{p-2}{2p}(J[\phi])^{\frac{p}{p-2}},
$$
where
$$
J[\phi]=\frac{\int_{\R^d} |\nabla \phi|^2 |x|^{2a} \,dx + \int_{\R^d} |\phi|^2 \,dx}{\left(\int_{\R^d} |\phi|^p \,dx\right)^{\frac 2 p}}.
$$
This immediately shows that
\begin{align} \label{gs1}
m=E[u]=\frac{p-2}{2p}\inf_{\phi \in H^{1,a}(\R^d)\backslash\{0\}}(J[\phi])^{\frac{p}{p-2}}
\end{align}
Moreover, since $u \in H^{1,a}(\R^d)$ is a ground state to \eqref{equ}, then
$$
\int_{\R^d} |\nabla u|^2 |x|^{2a} \,dx + \int_{\R^d} |u|^2 \,dx= \int_{\R^d} |u|^p \,dx.
$$
Therefore, by using \eqref{gs1}, we conclude that
\begin{align*}
m=E[u]&=\frac{p-2}{2p}\left(\int_{\R^d} |\nabla u|^2 |x|^{2a} \,dx + \int_{\R^d} |u|^2 \,dx\right)\\
&=\frac{p-2}{2p} (J[u])^{\frac{p}{p-2}}=\frac{p-2}{2p} \inf_{u \in H^{1,a}(\R^d) \backslash \{0\}} (J[\phi])^{\frac{p}{p-2}}.
\end{align*}
As a result, we know that $u \in H^{1,a}(\R^d)$ is a minimizer to the minimization problem
\begin{align} \label{wf}
\widetilde{m}:=\inf_{\phi \in H^{1,a}(\R^d) \backslash \{0\}} J[\phi],
\end{align}
where
$$
\widetilde{m}=\left(\frac{2pm}{p-2}\right)^{\frac{p-2}{p}}.
$$
Consequently, by Lemma \ref{lemr1} and the facts that $\|u^*\|_2=\|u\|_2$ and $\|u^*\|_p=\|u\|_p$, we derive that $u^* \in H^{1,a}(\R^d)$ is also a minimizer to \eqref{wf} and
\begin{align} \label{s0}
\int_{\R^d} |\nabla u^*|^2 |x|^{2a} \,dx = \int_{\R^d} |\nabla u|^2 |x|^{2a} \,dx.
\end{align}
This infers clearly that $u^* \in H^{1,a}(\R^d)$ is a also ground state to \eqref{equ}. At this point, employing Theorem \ref{thmdecay}, we get that $u^* \in C^{\infty}(\R^d  \backslash \{0\})$. 

In the following, we are going to prove that $u=u^*$. Since $u^*$ is radially symmetric, then we shall write $u^*=u^*(r)$ for $r=|x|$. And it holds that $u^* \in C^{\infty}(0, +\infty)$. 
Thanks to $u \in H^{1,a} (\R^d)$, by coarea formula, we know that, for any $t > 0$,
\begin{align} \label{pz}
\begin{split}
\left| \left\{ x \in \R^d : u(x) > t \right\}\right| 
&= \left|\left\{ x \in \R^d : |\nabla u(x)|=0\right\} \cap \left\{ x \in \R^d :  u(x)>t \right\} \right|  \\
& \quad + \left|\left\{ x \in \R^d : |\nabla u (x)|>0\right\} \cap \left\{ x \in \R^d :  u(x)>t \right\} \right| \\
&=\left|\left\{ x \in \R^d : |\nabla u(x)|=0\right\} \cap \left\{ x \in \R^d :  u(x) >t \right\} \right|  \\
& \quad + \int_t^{+\infty} \int_{u^{-1}(s)} |\nabla u|^{-1} \, \mathcal{H}_{d-1}(dx)ds,
\end{split}
\end{align}
where $|A|$ denotes the Lebesgue measure of the measurable set $A \subset \R^d$, $\mathcal{H}_{d-1}$ denotes $(d-1)$ dimensional Hausdorff measure and $u^{-1}(s):=\{x \in \R^d : u(x)=s\}$. 
Furthermore, since $u \in L^2(\R^d)$, we then get that, for any $t>0$,
\begin{align*} 
\left| \left\{ x \in \R^d : u(x) > t \right\}\right| \leq \frac{1}{t^2} \int_{\left\{ x \in \R^d : u(x) > t \right\}} |u|^2 \,dx<+\infty.
\end{align*}
This implies that, for any $t>0$, 
\begin{align} \label{bddu}
\left| \left\{ x \in \R^d : u(x) > t \right\}\right|<+\infty.
\end{align}
Coming back to \eqref{pz}, we then obtain that
$$
\int_t^{+\infty} \int_{u^{-1}(s)} |\nabla u|^{-1} \, \mathcal{H}_{d-1}(dx)ds < + \infty.
$$
It infers that, for a.e. $t > 0$,
$$
\int_{u^{-1}(t)} |\nabla u|^{-1} \, \mathcal{H}_{d-1}(dx)<+\infty.
$$
Similarly, we are able to conclude that, for a.e. $t > 0$,
$$
\int_{(u^*)^{-1}(t)} |\nabla u^*|^{-1} \, \mathcal{H}_{d-1}(dx)<+\infty.
$$
As a consequence of coarea formula, we are able to derive that
\begin{align} \label{c11}
\int_{\R^d} |\nabla u|^2 |x|^{2a} \,dx=\int_0^{+\infty} \int_{u^{-1}(t)} |\nabla u| |x|^{2a} \, \mathcal{H}_{d-1}(dx)dt.
\end{align}
Applying H\"older's inequality, we get that
$$
\int_{u^{-1}(t)}  |x|^{\mu} \, \mathcal{H}_{d-1}(dx) \leq \left(\int_{u^{-1}(t)} |\nabla u| |x|^{2a} \, \mathcal{H}_{d-1}(dx)\right)^{\frac 12} \left(\int_{u^{-1}(t)} \frac{|x|^{2(\mu-a)}}{|\nabla u|} \,\mathcal{H}_{d-1}(dx)\right)^{\frac 12},
$$
where $\mu \geq 1$ is a constant. Taking into account \eqref{c11}, we then obtain that
\begin{align} \label{s2}
\int_{\R^d} |\nabla u|^2 |x|^{2a} \,dx \geq  \int_0^{+\infty} \left(\int_{u^{-1}(t)}  |x|^{\mu} \, \mathcal{H}_{d-1}(dx)\right)^2\left(\int_{u^{-1}(t)} \frac{|x|^{2(\mu-a)}}{|\nabla u|}  \,\mathcal{H}_{d-1}(dx)\right)^{-1} \,dt.
\end{align}
Since $u^*$ is radially symmetry, then $|\nabla u^*|$ and $|x|$ are constants on $(u^*)^{-1}(t)$ for any $t>0$. This readily indicates that
$$
\int_{(u^*)^{-1}(t)}  |x|^{\mu} \, \mathcal{H}_{d-1}(dx) = \left(\int_{(u^*)^{-1}(t)} |\nabla u^*| |x|^{2a} \, \mathcal{H}_{d-1}(dx)\right)^{\frac 12} \left(\int_{(u^*)^{-1}(t)} \frac{|x|^{2(\mu-a)}}{|\nabla u^*|} \,\mathcal{H}_{d-1}(dx)\right)^{\frac 12}.
$$
Therefore, by coarea formula, we assert that
\begin{align} \label{s3}
\begin{split}
\int_{\R^d} |\nabla u^*|^2 |x|^{2a} \,dx &=\int_0^{+\infty} \int_{(u^*)^{-1}(t)} |\nabla u^*| |x|^{2a} \, \mathcal{H}_{d-1}(dx)dt \\
&=\int_0^{+\infty}  \left(\int_{(u^*)^{-1}(t)} |x|^{\mu} \, \mathcal{H}_{d-1}(dx)\right)^2\left(\int_{(u^*)^{-1}(t)} \frac{|x|^{2(\mu-a)}}{|\nabla u^*|} \,\mathcal{H}_{d-1}(dx)\right)^{-1} \,dt.
\end{split}
\end{align}
Since
$$
\int_{\left\{x \in \R^d : u(x)>t\right\}} |x|^{2(\mu-a)} \,dx=\int_{\left\{x \in \R^d : u^*(x)>t\right\}} |x|^{2(\mu-a)} \,dx,
$$
according to coarea formula, then
\begin{align} \label{s1}
\int_{u^{-1}(t)} \frac{|x|^{2(\mu-a)}}{|\nabla u|} \,\mathcal{H}_{d-1}(dx)=\int_{(u^*)^{-1}(t)} \frac{|x|^{2(\mu-a)}}{|\nabla u^*|} \,\mathcal{H}_{d-1}(dx).
\end{align}
Due to $\mu \geq 1$, then the function $\tau \mapsto \tau^{\mu}$ is strictly increasing and convex on $[0, +\infty)$. Moreover, since $u \in C^{\infty}(0, +\infty)$, then the boundary of $u^{-1}(t)$ is Lipschitz for any $t>0$. Invoking \cite[Theorem 2.1]{BBMP} and \eqref{bddu}, we then derive that
$$
\int_{u^{-1}(t)} |x|^{\mu} \, \mathcal{H}_{d-1}(dx) \geq \int_{(u^*)^{-1}(t)} |x|^{\mu} \, \mathcal{H}_{d-1}(dx).
$$
Combining this with \eqref{s0}, \eqref{s2}, \eqref{s3} and \eqref{s1} immediately indicates that
$$
\int_{(u^*)^{-1}(t)} |x|^{\mu} \, \mathcal{H}_{d-1}(dx) = \int_{u^{-1}(t)} |x|^{\mu} \, \mathcal{H}_{d-1}(dx).
$$
Hence, by \cite[Theorem 4.3]{BBMP}, we have that 
$$
\left\{ x\in \R^d : u(x) > t\right\} = \left\{ x\in \R^d : u^*(x) > t\right\}.
$$
This shows immediately that $u=u^*$, i.e $u$ is radially symmetric. 

Finally, we are going to show that $u$ is strictly decreasing in the radial direction. Since $u=u^*$, then it is radially symmetric and non-negative. This infers that $u$ satisfies the ordinary differential equation
\begin{align} \label{ode0}
u''+\frac{d-1+2a}{r}u'-\frac{u}{r^{2a}}+\frac{u^{p-1}}{r^{2a}}=0.
\end{align}
{Keeping in mind that $u$ is actually positive by the maximum principle,
let us now assume} by contradiction that there exists an interval $[r_1, r_2] \subset (0, +\infty)$ such that $u(r)=c>0$ for any $r \in [r_1, r_2]$, which means that $u''(r)=0$ and $u'(r)=0$ for any $r \in (r_1, r_2)$.  It then infers from \eqref{ode0} that $u(r)=u^{p-1}(r)$, i.e. $u(r)=1$ for any $r \in [r_1, r_2]$. Let $r_0 \in (r_1, r_2)$ be such that $u(r_0)=1$ and $u'(r_0)=0$. At this point, using the well-posedness of solution to the initial problem \eqref{ode0} with $u(r_0)=1$ and $u'(r_0)=0$, we then have that $u(r)=1$ for any $r \geq r_0$. This is a contradiction, because $u$ decays exponentially at infinity by Theorem \ref{thmdecay}. It then yields that $u$ is strictly decreasing in the radial direction. Thus the proof is complete.

\end{proof}


\section{Uniqueness} \label{uniqueness}

In this section, we are going to establish the uniqueness of ground states to \eqref{equ}. For this, we first need to prove the following result.

\begin{lem} \label{expdecay}
Let $u \in H^{1,a}(\R^d)$ be a ground state to \eqref{equ}. Then there exists $\widetilde{\delta}>0$ such that, for $1 \leq |\alpha| \leq 2$,
\begin{align} \label{decay10}
|D^{\alpha} u(x)| \lesssim e^{-\widetilde{\delta} |x|^{1-a}}, \quad |x|>R,
\end{align}
where $R>0$ is a constant.
\end{lem}
\begin{proof}
Since $u \in H^{1,a}(\R^d)$ is a ground state to \eqref{equ}, by Theorem \ref{symmetry}, then $u$ is radially symmetric and it satisfies the ordinary differential equation
\begin{align} \label{equ111}
u''+\frac{d-1+2a}{r}u'-\frac{u}{r^{2a}}+\frac{u^{p-1}}{r^{2a}}=0.
\end{align}
Let us write \eqref{equ111} as
\begin{align} \label{equ110}
\left(r^{d-1+2a} u'\right)'=r^{d-1} \left(u-u^{p-1}\right).
\end{align}
By integrating \eqref{equ110} on $[r_1, r_2]$ for $r_2>r_1>0$, we have that
$$
r_2^{d-1+2a} u'(r_2)-r_1^{d-1+2a} u'(r_1)=\int_{r_1}^{r_2} r^{d-1} \left(u-u^{p-1}\right) \,dr.
$$
Since $u$ decays exponentially at infinity by Theorem \ref{thmdecay}, then 
$$
r_2^{d-1+2a} u'(r_2)-r_1^{d-1+2a} u'(r_1) \to 0 \,\,\, \mbox{as}\,\, r_1, r_2 \to +\infty.
$$
It then indicates that $\lim_{r \to +\infty} r^{d-1+2a} u'(r)$ exists. Applying again the fact that $u$ decays exponentially at infinity, we further have that $ \lim_{r \to +\infty} r^{d-1+2a} u'(r)=0$. At this point, integrating \eqref{equ110} on $[r, +\infty)$ for $r>0$, we then conclude that
$$
r^{d-1+2a} u'(r)=-\int_r^{+\infty} \tau^{d-1} \left(u-u^{p-1}\right) \,d\tau.
$$
It then follows that there exists $R>0$ such that
$$
|u'(r)| \lesssim \int_r^{+\infty} \tau^{d-1} u\,d\tau, \quad r>R>0.
$$
Therefore, we know that \eqref{decay10} holds true for $|\alpha|=1$. Note that $u$ solves \eqref{equ111}. Then \eqref{decay10} holds true as well for $|\alpha|=2$. This completes the proof.
\end{proof}

In the sequel, we are going to adapt the approach due to Yanagida in \cite{Ya} to study the uniqueness of ground states to \eqref{equ}. Let $u \in H^{1,a}(\R^d)$ be a ground state to \eqref{equ}. Then, by Theorem \ref{symmetry} and the fact that $u$ is continuous at zero from Theorem \ref{thmdecay}, we are able to conclude that $u$ satisfies the ordinary differential equation
\begin{align} \label{equ1}
\left\{
\begin{aligned}
&u''+\frac{d-1+2a}{r}u'-\frac{u}{r^{2a}}+\frac{u^{p-1}}{r^{2a}}=0, \\
&u(0)>0, \quad \lim_{r \to \infty} u(r)=0.
\end{aligned}
\right.
\end{align}
To prove the uniqueness of ground states to \eqref{equ}, we shall introduce the corresponding Pohozaev quantity $J$ defined by
\begin{align} \label{defjru}
J(r, u):=\frac {A(r)}{2} (u')^2 + B(r) u' u + \frac {C(r)}{2}u^2-\frac{A(r)}{2}\frac{u^2}{r^{2a}} + \frac{A(r)}{p}\frac{u^{p}}{r^{2a}},
\end{align}
where $A, B$ and $C : (0, +\infty) \to \R$ are {unknown functions to be determined} later. Since $u$ satisfies \eqref{equ1}, then
it is simple to compute that
\begin{align} \label{defdj}
\begin{split}
\frac{d}{dr} J(r, u)&= \left(\frac{A_r}{2}-\frac{d-1+2a}{r} A + B\right) (u')^2 + \left(B_r-\frac{d-1+2a}{r} B+C\right)u'u \\
& \quad + \left(\frac{B}{r^{2a}} +\frac{C_r}{2}-\frac{A_r}{2r^{2a}}+\frac{aA}{r^{2a+1}}\right)u^2 + \left(\frac{A_r}{pr^{2a}}-\frac{2aA}{pr^{2a+1}}-\frac{B}{r^{2a}}\right)u^p.
\end{split}
\end{align}
{Let us now choose the functions $A, B$ and $C$ to satisfy the linear system}
$$
\frac{A_r}{2}-\frac{d-1+2a}{r} A + B=0, 
$$ 
$$
B_r-\frac{d-1+2a}{r}B +C=0,
$$ 
$$
\frac{A_r}{pr^{2a}}-\frac{2aA}{pr^{2a+1}}-\frac{B}{r^{2a}}=0.
$$
Therefore, we are able to calculate that
$$
A(r)=r^{\frac{4a+2(d-1+2a)p}{p+2}}, 
$$ 
$$
B(r)=\frac{2d-2+2a}{p+2}r^{\frac{4a+2(d-1+2a)p}{p+2}-1},
$$
$$
C(r)=\frac{2d-2+2a}{p+2} \left(d+2a -\frac{4a+2(d-1+2a)p}{p+2}\right)r^{\frac{4a+2(d-1+2a)p}{p+2}-2}.
$$
Going back to \eqref{defdj}, we then get that
\begin{align} \label{defg}
\frac{d}{dr} J(r,u)=G(r)u^2, 
\end{align}
where
\begin{align} \label{defg1}
\begin{split}
G(r)&=-\frac{(p-2)(d-1+a)}{p+2}r^{\frac{4a+2(d-1+2a)p}{p+2}-2a-1} \\
& \quad +\frac{2d-2+2a}{p+2} \left(d+2a -\frac{4a+2(d-1+2a)p}{p+2}\right) \\  
& \quad \times \left(\frac{2a+(d-1+2a)p}{p+2}-1\right)r^{\frac{4a+2(d-1+2a)p}{p+2}-3}.
\end{split}
\end{align}

\begin{lem} \label{lem1}
Let $u, v \in H^{1,a}(\R^d)$ be two ground states to \eqref{equ}. Then it holds that
\begin{align*} 
\frac{d}{dr} \left(\frac{v}{u}\right)=\frac{1}{r^{2a-1+d}u^2(r)} \int_0^r \tau^{d-1} \left(u^{p-1}-v^{p-1}\right) u v \,d\tau.
\end{align*}
\end{lem}
\begin{proof}
By Theorem \ref{symmetry}, we first know that $u$ and $v$ are radially symmetric and satisfy \eqref{equ1}. Now multiplying \eqref{equ1} by $r^{d-1+2a}v$ and integrating on $[s, r]$ for $r>s>0$, we have that
\begin{align} \label{u1}
\int_s^r u'' v\tau^{d-1+2a} \,d\tau + (d-1+2a)\int_s^r u'{\tau^{d-2+2a}} v \,d\tau=\int_s^r uv \tau^{d-1} \,d\tau -\int_s^r u^{p-1}v \tau^{d-1}\,d\tau.
\end{align}
As a consequence of integration by parts, we obtain that
\begin{align}  \label{u12}
&\int_s^r u' v' \tau^{d-1+2a}\,d\tau+\int_s^r uv \tau^{d-1} \,d\tau -\int_s^r u^{p-1}v \tau^{d-1}\,d\tau=u'(\tau)v(\tau){\tau^{d-1+2a}} \Big\lvert_{s}^r.
\end{align}
Reversing the roles of $u$ by $v$, we can also get that
\begin{align}  \label{u13}
&\int_s^r u' v' \tau^{d-1+2a}\,d\tau+\int_s^r uv \tau^{d-1} \,d\tau -\int_s^r v^{p-1}u\tau^{d-1}\,d\tau={v'(\tau)}u(\tau){\tau^{d-1+2a}} \Big\lvert_{s}^r.
\end{align}
Therefore, by combining \eqref{u12} and \eqref{u13}, we derive that
\begin{align*} 
&\left({u'(r)}v(r)- {v'(r)}u(r)\right){r^{2a-1+d}}+\int_s^r \tau^{d-1} \left(u^{p-2}-v^{p-2}\right) u v \,d\tau ={u'(s)}v(s){s^{2a-1+d}}-{v'(s)}u(s){s^{2a-1+d}}.
\end{align*}
Invoking Theorem \ref{thmdecay}, we know that
\begin{align} \label{decay11}
\begin{split}
u'(r)&=-\frac{u^{p-1}(0)-u(0)}{d} r^{1-2a}+o(r^{1-2a}), \quad r \to 0, \\
v'(r)&=-\frac{v^{p-1}(0)-v(0)}{d} r^{1-2a}+o(r^{1-2a}), \quad r \to 0.
\end{split}
\end{align}
This shows that
$$
\lim_{s \to 0}{u'(s)}v(s){s^{d-1+2a}}=\lim_{s \to 0}{v'(s)}u(s){s^{d-1+2a}}=0.
$$
Accordingly, taking the limit as $s \to 0$, we obtain that
$$
\left({v'(r)}u(r)-{u'(r)}v(r) \right){r^{2a-1+d}}=\int_0^r \tau^{d-1} \left(u^{p-1}-v^{p-1}\right) u v \,d\tau.
$$
This implies readily the desired conclusion and the proof is complete.
\end{proof}

\begin{lem} \label{lem2}
Let $u, v \in H^{1,a}(\R^d)$ be two ground states to \eqref{equ} with $u(0)<v(0)$. If $J(r,u) \geq 0$ for any $r>0$, then it holds that
$$
\frac{d}{dr} \left(\frac{v}{u}\right)<0, \quad \forall \, r>0.
$$
\end{lem}
\begin{proof}
We shall assume by contradiction that the {conclusion} does not hold. Define $w={v}/{u}$. Since $u(0)<v(0)$, by Lemma \ref{lem1}, then there exists $r_0>0$ such that $w'(r)<0$ for any $0<r<r_0$. Define
$$
r_*:=\sup \left\{r>0: w'(\tau)<0, \,\, 0<\tau < r\right\}.
$$
It follows from the contradiction that $r_*<+\infty$. As a consequence, we get that $w'(r_*)=0$ and $w'(r)<0$ for any $0<r<r_*$. This implies that $w(r_*)<1$. If not, then we shall assume that $w(r^*) \geq 1$. From Lemma \ref{lem1}, we then know that $w'(r_*)>0$. This is impossible. Define
\begin{align} \label{defx}
X(r):=w^2J(r, u)-J(r, v), \quad r>0.
\end{align}
It is simple to calculate by \eqref{defjru} that
\begin{align} \label{defx1}
X(r)=\frac{A(r)}{2} \left(\frac{v^2(u')^2}{u^2}-(v')^2\right)+B(r) \left(\frac{u' v}{u}-v'\right)v+\frac{A(r)}{pr^{2a}} \left(u^{p-2}-v^{p-2}\right) v^2.
\end{align}
Observe that
\begin{align} \label{c0}
\frac{2a+(d-1+2a)p}{p+2}-a>0,
\end{align}
\begin{align} \label{c00}
1-2a+\frac{2a+(d-1+2a)p}{p+2}>0.
\end{align}
Therefore, applying \eqref{decay11}, we have that $\lim_{r \to 0} X(r)=0$. 
Since $w(r_*)<1$, $w'(r_*)=0$ and $u, v>0$, then
\begin{align} \label{c1}
X(r_*)=\frac{A(r_*)}{pr_*^{2a}} \left(u^{p-2}(r_*)-v^{p-2}(r_*)\right) v^2(r_*)>0.
\end{align}
Furthermore, applying \eqref{defg}, we see that
\begin{align} \label{defx2}
\frac{dX}{dr}=2ww'J(r,u)+w^2 \frac{d}{dr} J(r,u)-\frac{d}{dr}J(r,v)=2ww'J(r,u).
\end{align}
Since $J(r, u) \geq 0$ for any $r>0$ and $w'(r)<0$ for any $0<r<r^*$, then
$$
\frac{dX}{dr} \leq 0 , \quad 0<r<r_*,
$$
We now reach a contradiction from \eqref{c1} and the fact that $\lim_{r \to 0} X(r)=0$. This completes the proof.
\end{proof}

\begin{lem} \label{lem3}
Let $u \in H^{1,a}(\R^d)$ be a ground state to \eqref{equ}. Then it holds that $J(\cdot, u) \not\equiv 0$ and $J(r,u) \geq 0$ for any $r>0$.
\end{lem}
\begin{proof}
Observe first that
\begin{align} \label{u1110}
d+2a -\frac{4a+2(d-1+2a)p}{p+2}>0,
\end{align}
due to $p<2^*_a$. In addition, it holds that
\begin{align} \label{u111}
\frac{2a+(d-1+2a)p}{p+2}-1>0.
\end{align}
From \eqref{defg1}, we then find that $G(r)>0$ for any $r>0$ small enough. It then gives that $J(\cdot, u) \not\equiv 0$. Next we are going to demonstrate that $J(r,u) \geq 0$ for any $r>0$. 
It follows from \eqref{decay11}, \eqref{c0} and \eqref{c00} that $\lim_{r \to 0} J(r,u)=0$.  
Since $G(r)>0$ for any $r>0$ small enough, then there exists $r_1>0$ such that $J(r, u)>0$ for any $0<r<r_1$. Define
$$
r_*:=\sup \left\{r>0 : J(r, u) \geq 0, \,\, 0<\tau<r\right\}.
$$
If $r^*=+\infty$, then the proof is {complete}. If not, then $r^*<+\infty$. This then indicates that $J(r,u)>0$ for any $0<r<r^*$ and
$$
J(r^*, u)=0, \quad \frac{d}{dr} J(r^*, u)<0.
$$
In view of \eqref{defg1}, \eqref{u1110} and \eqref{u111}, we are able to show that there exists a unique $r_0 > 0$ such that
$$
\frac{d}{dr} J(r, u)>0, \quad \forall \, 0<r<r_0, \quad \frac{d}{dr} J(r, u)<0, \quad \forall \, r>r_0.
$$
Then we conclude that
$$
\frac{d}{dr} J(r, u)<0, \quad \forall \, r>r^*.
$$
This is impossible, because $J(r,u) \to 0$ as $r \to +\infty$ by Theorem \ref{thmdecay} and Lemma \ref{expdecay}. 
Hence the proof is complete.
\end{proof}

\begin{thm} \label{thm1}  
It holds that there exists only one ground state to \eqref{equ} in $H^{1,a}(\R^d)$.
\end{thm}
\begin{proof}
Suppose by contradiction that there exist two ground states $u$ and $v$ in $H^1_a(\R^d)$ to \eqref{equ}. Then we know that $u$ and $v$ satisfy \eqref{equ} with $J(r, u) \to 0$ and $J(r ,v) \to 0$ as $r \to +\infty$. Without restriction, we assume that $u(0)<v(0)$. By Lemma \ref{lem3}, we then have that $J(\cdot, u) \not\equiv 0$ and $J(\cdot, v) \not\equiv 0$. In addition, we get that $J(r, u) \geq 0$ and $J(r, v) \geq 0$ for any $r>0$. Let $X$ be defined by \eqref{defx}. We find that $\lim_{r \to 0} X(r)=0$. Moreover, by utilizing Theorem \ref{thmdecay} and Lemma \ref{expdecay}, we obtain that $\lim_{r \to +\infty} X(r)=0$. However, it follows from Lemma \ref{lem2} and \eqref{defx2} that
$$
\frac{dX}{dr} \leq 0, \quad \frac{dX}{dr} \not\equiv 0.
$$
Obviously, this is a contradiction, because there holds that $\lim_{r \to 0} X(r)=0$ and $\lim_{r \to +\infty} X(r)=0$. Then the proof is {complete}.
\end{proof}

\section{Non-degeneracy} \label{nond}

The aim of this section is to establish the non-degeneracy of the ground state $u$ to \eqref{equ} in $H_{rad}^{1,a}(\R^d)$. For this, we first need to discuss asymptotic behaviors of solutions to the corresponding linearized equation. 

\begin{lem} \label{lemdecay}
Let $v \in H^{1,a}_{rad}(\R^d)$ be a nontrivial solution to the linearized equation
\begin{align} \label{equ11}
-\nabla \cdot \left(|x|^{2a} \nabla v \right) + v=(p-1) u^{p-2}v.
\end{align}
Then it holds {that} $v \in C^{\infty}(\R^d \backslash \{0\})$ and
$$
v(r) \sim 1, \quad v'(r) =o\left(\frac 1 r\right), \quad r \to 0.
$$
Moreover, if $0<a<1/2$, then it holds that
$$
v(r) \sim r^{-\frac{d-1}{2}} e^{-\frac{1}{1-a}r^{1-a}}, \quad v'(r) \sim  r^{-\frac{d-1}{2}-a} e^{-\frac{1}{1-a}r^{1-a}} +o\left(\frac{1}{r^{\frac{d-1}{2}+2a}}e^{-\frac{1}{1-a}r^{1-a}}\right), \quad r \to +\infty.
$$
\end{lem}
\begin{proof}
Since $u \in C^{\infty}(\R^d \backslash \{0\})$ by Theorem \ref{thmdecay}, by using the standard bootstrap procedure, then we have that $v \in C^{\infty}(\R^d \backslash \{0\})$. Let us first consider the asymptotic behaviors of $v$ and $v'$ as $r$ approach zero. Since $v \in H^{1,a}_{rad}(\R^d)$ is a solution to \eqref{equ11}, then it satisfies the ordinary differential equation
\begin{align} \label{ode111}
\left(r^{d-1+2a} v'\right)'-r^{d-1} \left(1-(p-1)u^{p-2}\right)v=0.
\end{align}
Define $w(r)=v(1-r)$ for $r > 0$. It then follows that $w$ satisfies the equation
\begin{align} \label{equ112}
\left((1-r)^{d-1+2a} w'\right)'-(1-r)^{d-1} \left(1-(p-1)h(r)\right)w=0,
\end{align}
where $h(r):=\left(u(1-r)\right)^{p-2}$ for $r>0$. It is clear that \eqref{equ112} is equivalent to the binary system of the form
\begin{align} \label{sys}
\left\{
\begin{aligned}
&w'=\frac{1}{(1-r)^{d-1+2a}} z, \\
& z'=(1-r)^{d-1} \left(1-(p-1)h(r)\right)w.
\end{aligned}
\right.
\end{align}
Observe that
$$
\int^{1} (1-r)^{d-1} \left|1-(p-1)h(r) \right| \,dr<+\infty,
$$
$$
\int^{1} (1-r)^{d-1} \left|1-(p-1)h(r)\right| \int^r \frac{1}{(1-s)^{d-1+2a}} \,ds dr<+\infty,
$$
where the notation
$
\int^{\tau} f(r) \,dr
$
denotes the integral of $f$ over a left neighborhood of $\tau \leq + \infty$. In addition, it holds that
$$
\int^{1} \frac{1}{(1-r)^{d-1+2a}} \,dr=+\infty.
$$
Now, by taking into account \cite[Lemma 9.3]{H}, we derive that there exist two linearly independent solutions $w_1$ and $w_2$ to \eqref{equ112} such that
$$
w_1 \sim 1, \quad w_2 \sim \frac{1}{(1-r)^{d-2+2a}}, \quad r \to 1.
$$
Meanwhile, we derive that
$$
z_1=o\left((1-r)^{d-2+2a}\right), \quad z_2 \sim 1,  \quad r \to 1.
$$
At this point, applying \eqref{sys}, we conclude that
$$
w_1'=o\left(\frac {1}{1-r}\right), \quad w_2' \sim \frac{1}{(1-r)^{d-1+2a}}, \quad r \to 1.
$$
Since $v \in H^{1,a}_{rad}(\R^d)$, by making a change of variable with $1-r \mapsto r$, then we have that 
$$
v(r) \sim 1, \quad v'(r) =o\left(\frac 1 r\right), \quad r \to 0.
$$
This yields the desired conclusion.

Next, we shall investigate the asymptotic behaviors of $v$ and $v'$ as $r$ goes to infinity. Define 
\begin{align} \label{zv}
\zeta(r):=r^{\frac{d-1+2a}{2}} v(r), \quad r>0.
\end{align}
Since $v$ solves \eqref{equ112}, then we get that $\zeta$ satisfies the equation
\begin{align} \label{equz}
\zeta''-\frac{d-1+2a}{2} \left(\frac{d-1+2a}{2}+1\right)\frac{\zeta}{r^2}-\frac{\zeta}{r^{2a}} +\frac{p-1}{r^{2a}} u^{p-2}\zeta=0.
\end{align}
Let $\eta_1$ and $\eta_2$ be two linearly independent solutions to the equation
$$
\eta''-\frac{1}{r^{2a}} \eta=0.
$$
Observe that
$$
\int^{+\infty} \frac{1}{r^a} \,dr=+\infty, \quad \frac{a(2-a)}{a}\int^{+\infty} \frac{1}{r^{2-a}} \,dr<+\infty.
$$
It then follows from \cite[Exercise 9.6]{H} that
$$
\eta_1 \sim r^a e^{-\frac{1}{1-a}r^{1-a}}, \quad \eta_2 \sim  r^a e^{\frac{1}{1-a}r^{1-a}}, \quad r \to +\infty.
$$
In addition, it holds that
$$
\eta_1' \sim e^{-\frac{1}{1-a}r^{1-a}}, \quad \eta_2' \sim e^{\frac{1}{1-a}r^{1-a}}, \quad r \to +\infty.
$$
Since $u$ decays exponentially at infinity by Theorem \ref{thmdecay} and $0<a<1/2$, then we arrive at
$$
\int^{+\infty} \left|\frac{d-1+2a}{2} \left(\frac{d-1+2a}{2}+1\right) \frac{1}{r^2}-\frac{p-1}{r^{2a}}u^{p-2}\right| r^{2a}\,dr<+\infty.
$$
It then follows from \cite[Theorem 9.1]{H} that there exist two linearly independent solutions to \eqref{equz} such that
$$
\zeta_1 \sim \eta_1, \quad \zeta_2 \sim \eta_2, \quad r \to + \infty,
$$
$$
\frac{\zeta_1'}{\zeta_1} \sim \frac{1}{r^a}+o\left(\frac{1}{r^{2a}}\right), \quad \frac{\zeta_2'}{\zeta_2} \sim \frac{1}{r^a}+o\left(\frac{1}{r^{2a}}\right), \quad r \to +\infty.
$$
Since $v \in H^{1,a}_{rad}(\R^d)$, by applying \eqref{zv}, then we obtain the desired conclusion. This completes the proof.
\end{proof}

\begin{lem} \label{sign}
Let $v \in H_{rad}^{1,a}(\R^d)$ be a nontrivial solution to \eqref{equ11}. Then there exists a unique $r_0>0$ such that $v(r_0)=0$ and $v$ changes sign at $r_0$.
\end{lem}
\begin{proof}
Define
$$
\mathcal{B}(\varphi, \psi):=\int_{\R^d} |x|^{2a} \nabla \varphi \cdot \nabla \psi + \varphi \psi -(p-1)u^{p-2}\varphi \psi \,dx, \quad \forall\,\, \varphi, \psi \in H^{1,a}_{rad}(\R^d).
$$
Then we know that there exists a unique self-adjoint operator $T: \mathcal{D}(T) \subset L^2_{rad}(\R^d)\to L^2_{rad}(\R^d)$ such that $T \varphi = \phi$ and
$$
\mathcal{D}(T) :=\left\{ \varphi \in H^{1,a}(\R^d) : \phi \in L^2_{rad}(\R^d), \mathcal{B}(\varphi, \psi)=\langle \phi, \psi \rangle, \forall \,\, \psi \in H^{1,a}_{rad}(\R^d) \right\}.
$$
Since $v \in H_{rad}^{1,a}(\R^d)$ is a {nontrivial} solution to \eqref{equ11}, then $\mathcal{B}(v, \psi)=0$ for any $\psi \in H^{1,a}_{rad}(\R^d)$. It then follows that $T v=0$. Since $Tv=0$ is is non-oscillatory, by \cite[Theorem 14.9]{W}, we then conclude that $T$ is bounded from below. Let $\varphi \in H_{rad}^{1,a}(\R^d)$ be such that $(T - \lambda) \varphi=0$ for $\lambda<1$. Since $(T - \lambda) \varphi=0$ is non-oscillatory, by \cite[Theorem 14.9]{W}, we then get that $\inf \sigma_{ess}(T) \geq 1$, where $ \sigma_{ess}(T)$ denotes the essential spectrum of the operator $T$. Since $Tv=0$, then $0$ is an eigenvalue of $T$. Moreover, we know that it is an isolated eigenvalue with finite multiplicity, because of $\sigma_{ess}(T) \geq 1$. Observe that
\begin{align} \label{equ13}
v''+\frac{d-1+2a}{r}v'-\frac{v}{r^{2a}}+\frac{p-1}{r^{2a}}u^{p-2}v=0,
\end{align}
\begin{align} \label{equ14}
u''+\frac{d-1+2a}{r}u'-\frac{u}{r^{2a}}+\frac{u^{p-1}}{r^{2a}}=0.
\end{align}
It is not hard to compute by integration by parts, Theorem \ref{thmdecay} and Lemmas \ref{expdecay} and \ref{lemdecay} that
$$
\int_0^{+\infty} \left(v'' + \frac{d-1+2a}{r}v' \right) r^{d-1+2a} u \,dx=\int_0^{+\infty} \left(u'' + \frac{d-1+2a}{r}u' \right) r^{d-1+2a} v \,dx.
$$
Therefore, by combining \eqref{equ13} and \eqref{equ14}, we know that
$$
\int_0^{+\infty}r^{d-1} \left(1-(p-1)u^{p-2}\right) uv \,dx=\int_0^{+\infty} r^{d-1} \left(1-u^{p-2}\right) uv \,dx.
$$
It then infers that
\begin{align} \label{id1}
\int_0^{+\infty}r^{d-1} u^{p-1}v \,dr=0.
\end{align}
Hence we know that $v$ changes sign in $(0, +\infty)$, because of $u>0$. It then follows from \cite[Theorem 14.10]{W} that $v$ is not an eigenfunction of the smallest eigenvalue of $T$. As a consequence, we derive that $T$ admits at least one negative eigenvalue. Indeed, we are able to show that there exists only one negative eigenvalue for the operator $T$. Contrarily, we assume that there exist two negative eigenvalues $\lambda_2<\lambda_1<0$ corresponding to the operator $T$. Therefore, there exist {nontrivial} $\varphi_1, \varphi_2 \in H_{rad}^{1,a}(\R^d)$ such that
$$
T \varphi_1=\lambda_1 \varphi_1, \quad T \varphi_2=\lambda_2 \varphi_2.
$$
Then we arrive at
\begin{align}\label{e1}
\mathcal{B}(\varphi_1, \varphi_2)=\lambda_1 \langle \varphi_1, \varphi_2 \rangle,
\quad \mathcal{B}(\varphi_2, \varphi_1)=\lambda_1 \langle \varphi_2, \varphi_1 \rangle.
\end{align}
Note that $\mathcal{B}(\varphi_1, \varphi_2)=\mathcal{B}(\varphi_2, \varphi_1)$. It then follows that $\langle \varphi_1, \varphi_2 \rangle=0$ and $\mathcal{B}(\varphi_1, \varphi_2)=0$, because of $\lambda_1 \neq \lambda_2$. Since $u \in H^{1,a}(\R^d)$ is a ground state to \eqref{equ}, then its Morse index is one. In addition, we know that $\mathcal{B}(u,u)<0$. This implies that $\mathcal{B}(\varphi, \varphi) \geq 0$ for any $\varphi \in u^{\bot}$, where
$$
u^{\bot}:=\left\{ u \in H_{rad}^{1,a}(\R^d): \langle \varphi, u \rangle=0\right\}.
$$
We now take $\alpha_1,\alpha_2 \in \R$ such that $\alpha_1 \varphi_1 + \alpha_2 \varphi_2 \in u^{\bot}$. Since $\mathcal{B}(\varphi_1, \varphi_2)=0$, by \eqref{e1}, then
\begin{align*}
0 \leq \mathcal{B}(\alpha_1 \varphi_1 + \alpha_2 \varphi_2, \alpha_1 \varphi_1 + \alpha_2 \varphi_2 )&=\alpha_1^2\mathcal{B}(\varphi_1, \varphi_1)+\alpha_2^2\mathcal{B}(\varphi_2, \varphi_2)\\
&=\alpha_1^2 \langle T \varphi_1, \varphi_1 \rangle +\alpha_2^2 \langle T \varphi_2, \varphi_2 \rangle \\
&=\alpha_1^2 \lambda_1 \langle \varphi_1, \varphi_1 \rangle+\alpha_2^2 \lambda_2 \langle \varphi_2, \varphi_2 \rangle<0,
\end{align*}
which is impossible. Thereby, we conclude that $T$ has only one negative eigenvalue. It then infers that $0$ is the second eigenvalue of the operator $T$. At this point, using \cite[Theorem 14.10]{W}, we then derive that $v$ has exactly one zero in $(0, +\infty)$. This completes the proof.
\end{proof}

\begin{thm} \label{nd}
It holds that $Ker [\mathcal{L}_+]=0$ in $H^{1,a}_{rad}(\R^d)$.
\end{thm}
\begin{proof}
Let us argue by contradiction that $Ker [\mathcal{L}_+] \neq 0$ in $H^{1,a}_{rad}(\R^d)$. This means that there exists a {nontrivial} $v \in H_{rad}^{1,a}(\R^d)$ satisfying \eqref{equ11}. By Lemma \ref{sign}, we know that $v$ has exactly one zero in $(0, +\infty)$. Without restriction, we may assume that there exists $r_0>0$ such that $v(r)<0$ for any $0<r<r_0$ and $v(r)>0$ for any $r>r_0$. Now multiplying \eqref{equ13} by $r^{d-1+2a} ru'$ and integrating on $(0, +\infty)$ leads to
$$
\int_0^{+\infty} \left(r^{d-1+2a}v'\right)' ru'-r^{d}\left(1-(p-1)u^{p-2}\right)vu'\,dr=0.
$$ 
Making use of integration by parts together with Theorem \ref{thmdecay} and Lemmas \ref{expdecay} and \ref{lemdecay}, we then have that
\begin{align} \label{i1}
\int_0^{+\infty} \left(r^{d-1+2a}v'\right) (ru')'+r^{d}\left(1-(p-1)u^{p-2}\right)vu' \,dr=0.
\end{align}
Invoking \eqref{equ14}, we know that
$$
(ru')'=ru''+u'=(2-d-2a)u'+r\left(\frac{u}{r^{2a}}-\frac{u^{p-1}}{r^{2a}} \right),
$$
$$
 \left(r^{d-1+2a}u'\right)'=r^{d-1}\left(u-u^{p-1} \right).
$$
As a consequence, using again integration by parts along with Theorem \ref{thmdecay} and Lemmas \ref{expdecay} and \ref{lemdecay}, we get that
\begin{align*}
\int_0^{+\infty} \left(r^{d-1+2a}v'\right) (ru')' \,dr&=(2-d-2a)\int_0^{+\infty} \left(r^{d-1+2a}u'\right)v'\,dr+\int_0^{+\infty}r^d \left(1-u^{p-2}\right)uv'\,dr \\
&=-(2-d-2a)\int_0^{+\infty} r^{d-1}\left(u-u^{p-1}\right)\,dr+\int_0^{+\infty}r^d \left(1-u^{p-2}\right)uv'\,dr.
\end{align*}
Taking into account \eqref{i1}, we then obtain that
$$
-(d-2+2a) \int_0^{+\infty}r^{d-1}\left(u-u^{p-1} \right) v\,dr + \int_0^{+\infty}r^d \left(\left(1-u^{p-2}\right)uv' +\left(1-(p-1)u^{p-2}\right)vu'\right)\,dr=0.
$$
Moreover, integrating by parts and employing Theorem \ref{thmdecay} and Lemmas \ref{expdecay} and \ref{lemdecay}, we see that
\begin{align*}
\int_0^{+\infty}r^d \left(1-u^{p-2}\right)uv' \,dr
&=-d\int_0^{+\infty} r^{d-1}\left(1-u^{p-2}\right)uv \,dr-\int_0^{+\infty}r^d \left(1-u^{p-2}\right)u'v \,dr \\
& \quad +(p-2)\int_0^{+\infty}r^d u^{p-2}u'v \,dr \\
&=-d\int_0^{+\infty} r^{d-1}\left(u-u^{p-1}\right)v \,dr-\int_0^{+\infty}r^d \left(1-(p-1)u^{p-2}\right)u'v \,dr.
\end{align*}
Therefore, we derive that
\begin{align} \label{id2}
\int_0^{+\infty}r^{d-1}\left(u-u^{p-1}\right)v \,dr=0.
\end{align}
{Combining} \eqref{id1} and \eqref{id2}, we have that, for any $\alpha \in \R$,
$$
\alpha \int_0^{+\infty}r^{d-1} u^{p-1} v \,dr+\int_0^{+\infty}r^{d-1}\left(u-u^{p-1}\right)v \,dr=0,
$$
Consequently, it holds that
$$
\int_0^{+\infty} r^{d-1} u^{p-1}\left(\alpha + u^{2-p}-1\right) v \,dr=0.
$$
Since $u$ is strictly decreasing on $(0, +\infty)$, then $u^{2-p}(r)-1<u^{2-p}(r_0)-1$ for any $r<r_0$ and $u^{2-p}(r)-1>u^{2-p}(r_0)-1$ for any $r>r_0$.
At this point, choosing $\alpha=1-u^{2-p}(r_0)$ and noting that $v(r)<0$ for any $0<r<_0$ and $v(r)>0$ for any $r>r_0$, we then get that
$$
\int_0^{+\infty} r^{d-1} u^{p-1}\left(\alpha + u^{2-p}-1\right) v \,dr>0.
$$
We then reach a contradiction. This completes the proof.
\end{proof}

\end{document}